\documentclass[12pt, reqno]{amsart}
\usepackage{fullpage}
\usepackage{amsmath,amssymb,amsfonts,amsthm}
\usepackage{tikz}
\usepackage{enumitem}
\usepackage{graphicx}
\usepackage{pgfplots}
\usetikzlibrary{arrows}
\usepackage{thm-restate}

\usepackage[colorlinks=true,linkcolor=blue, allcolors=blue]{hyperref}
\usepackage[backend=biber, style=alphabetic, isbn=false, url=false, doi=false]{biblatex}
\makeglossary
\usepackage{cleveref}
\usepackage{MnSymbol}
\usepackage{csquotes}
\usepackage{enumitem}
\usepackage{hyperref}
\usepackage{verbatim}
\usepackage{tikz}
\usepackage{tikz-cd}
\usepackage{color}
\usepackage{graphicx}

\newcommand{\Spec}{\mathrm{Spec}}

\newcommand{\f}{\mathcal{F}}
\newcommand{\ls}{\mathcal{L}}

\newcommand{\pp}{\mathbb{P}}

\newcommand{\nn}{\mathcal{N}}
\newcommand{\ga}{\mathfrak{g}}

\newcommand{\cG}{\mathcal{G}}

\newcommand{\cha}{\textrm{char}}
\newcommand{\Perv}{\mathrm{Perv}}
\newcommand{\tame}{\mathrm{tame}}

\newcommand{\remove}[1]{ }

\newcommand{\fg}{\mathfrak{g}}

\newcommand{\maps}{\rightarrow}

\newcommand{\ind}{\Gamma_B^G}

\newtheoremstyle{mybold}
{3pt}
{3pt}
{}
{0pt}
{\bfseries}
{.}
{.5em}
{}

\newtheorem{theorem}{Theorem}[section]
\newtheorem{lemma}[theorem]{Lemma}
\newtheorem{definition}[theorem]{Definition}

\newtheorem{proposition}[theorem]{Proposition}
\newtheorem{example}[theorem]{Example}

\theoremstyle{mybold}
\newtheorem{remark}[theorem]{Remark}

\pgfplotsset{compat=1.18}

\title{Character sheaves in characteristic $p$ have nilpotent singular support}
\author{Kostas I. Psaromiligkos}
\date{}
\address{Université Clermont Auvergne, CNRS, LMBP, F-63000 Clermont-Ferrand, France. email: konstantinos.psaromiligkos@uca.fr.}
\keywords{}

\begin{document}
	
	\begin{abstract}
		We prove that character sheaves of a reductive group defined over any characteristic have nilpotent singular support, partially extending the work of \cite{Mirkovic1988} and \cite{Ginzburg1989} to positive characteristic. We do this by introducing a category of tame perverse sheaves and studying its properties. 
	\end{abstract}
	
	\maketitle
	
	\section{Introduction}
	
	\subsection{Summary}
	Character sheaves were constructed by Lusztig in a series of papers \cite{Lusztig1985a}, \cite{Lusztig1985b}, \cite{Lusztig1985c}, \cite{Lusztig1986}, \cite{Lusztig1986a} and provide a geometric analog of character theory. 
	
	Despite their utility, the theory of character sheaves can be fairly technical. Shortly after their construction, Lusztig and Laumon conjectured that for reductive groups defined over a field $k$ of characteristic zero, there is a simple characterization of character sheaves in terms of their singular support. 
	
	Let $G$ be a reductive group over $k$ and $\nn\subseteq \fg$ be the nilpotent cone of $G$. It can be embedded in $\ga^*$ using the Killing form, and then we have $G\times \nn\subseteq T^*G\cong G\times \ga^*$. Mirkovic and Vilonen \cite{Mirkovic1988} and independently Ginzburg \cite{Ginzburg1989} proved that an irreducible $G$-equivariant perverse sheaf $\f$ on $G$ is a character sheaf if and only if $SS(\f)\subseteq G\times \nn$. Ginzburg employed the Riemann-Hilbert correspondence and the classical notion of singular support for $D$-modules, whereas Mirkovic and Vilonen used the microlocal definition of singular support of constructible sheaves on varieties over characteristic $0$ defined in \cite{Kashiwara1990}, which circumvents the use of $D$-modules. 
	
	For $k$ a field of characteristic $p$, there was no satisfactory notion of singular support until 2015, when Beilinson \cite{Beilinson2016} constructed the singular support of a constructible \'etale sheaf with torsion coefficients over any field. Beilinson's notion was extended to $l$-adic sheaves by Barrett \cite{Barrett2023}.
	
	In this paper, inspired by the construction of the tame site \cite{Huebner2021}, we define a category of {\em tame} perverse sheaves and study their functorial properties. This notion captures most sheaves used in the construction of character sheaves, and tame perverse sheaves behave similarly enough to the characteristic $0$ case. We then adapt Mirkovic and Vilonen's proof to show the following, without assumptions on $\cha(k)$.
	
	\begin{theorem}\label{main}
		Let $G$ be a reductive group over an algebraically closed field $k$. Then the singular support of a character sheaf is a subvariety of $G\times \nn$.
	\end{theorem}
	
	\subsection{Outline}
	
	In Section 2, we recall the basic properties of the singular support that we will need.
	
	In Section 3, we define a general notion of tame perverse sheaves, and show that the category they form is suitable for our purposes, using Kerz and Schmidt's results about tameness of \'etale coverings \cite{Kerz2010}. 
	
	Upon trying to adapt the proof of Mirkovic and Vilonen to positive characteristic, we stumble in the following problem. Mirkovic and Vilonen use in an essential way the conormality of singular support to its base in characteristic $0$, which does not remain true in characteristic $p$. We show that conormality of the singular support is true for tame perverse sheaves, upon restrictions on the geometry of the ramification divisor.
	
	In Section 4, we prove Theorem \ref{main}.
	
	\subsection{Acknowledgements}
	
	I thank Sasha Beilinson, Siddharth Mahendraker, Luca Migliorini, Châu Ngô, Simon Riche, Takeshi Saito, Colton Sandvik, Will Sawin and Tong Zhou for their interest, comments, and valuable discussions. 
	
	In particular, I am greatly thankful to Sasha Beilinson for an abundance of helpful discussions regarding the notion of singular support, and Tong Zhou for helping correct a previous version of Lemma \ref{conormality}. 
	
	This work was part of my PhD thesis at the University of Chicago, advised by Châu Ngô, to whom I am also thankful for his mentorship and guidance throughout this process.
	
	This project has received funding from the European Research Council (ERC) under the European Union’s Horizon 2020 research and innovation programme (grant agreement No. 101002592), and by a Graduate Research Fellowship from the Onassis foundation.
	
	\section{Review on singular support}
	
	\subsection{Definition and properties of singular support}
	For us, \enquote{variety} means \enquote{$k$-scheme of finite type}. We call a variety $X$ smooth if it is smooth relative to $\Spec k$, and then $T^*X$ denotes the cotangent bundle relative to $k$.
	
	A \emph{test pair} for a variety $X$ is a correspondence of the form $X\xleftarrow{h} U\xrightarrow{f} Y$, where $U, Y$ are also varieties. Recall that a morphism $f:X\rightarrow Y$ of smooth schemes induces a map of vector bundles $df: T^*Y\times_Y X\rightarrow T^*X$. The following definitions are from \cite[\S 1]{Beilinson2016}, see also \cite[\S 1.3]{Barrett2023}.
	
	\begin{definition}
		Let $h:U\rightarrow X$ be a morphism where $U,X$ are smooth and $C\subseteq T^*X$ be a closed conical subset of the cotangent bundle of $X$.
		
		The map $h$ is called $C$-transversal at a geometric point $u\rightarrow U$ if for every $v\in C_{h(u)}$ nonzero we have that $dh(v)\in T_u^*U$ is also non-zero. If it is $C$-transversal at every geometric point, we call the map $C$-transversal.
		
		If $h$ is $C$-transversal, we define $C_U=C\times_X U$, and then the restriction of $dh$ to $C_U$ is a finite map by \cite[Lemma 1.2(ii)]{Beilinson2016}. Its image $h^{\circ}C$ is then a closed conical subset of $T^*U$, which unravelling the definition has fiber over $u$ given by $(h^{\circ}C)_u:=dh(C_{h(u)})\subseteq T_u^*U$.	
	\end{definition}
	
	\begin{remark}
		A smooth morphism is $C$-transversal for any $C$. A pair $(h,f)$ is transversal with respect to the zero section if and only if $f$ is smooth, and $T^*X$-transversal if and only if $h\times f$ is smooth, see \cite[Example 1.2]{Beilinson2016}.
	\end{remark}
	
	\begin{definition}
		Let $f:U\rightarrow Y$ be a morphism where $X,Y$ are smooth and $C\subseteq T^*Y$ be a closed conical subset of the cotangent bundle of $Y$.
		
		The map $f$ is called $C$-transversal at a geometric point $u\rightarrow U$ if for every nonzero $v\in T^*_{f(u)}Y$ we have $df(v)\notin C_u$. If it is $C$-transversal at every geometric point, we call the map $C$-transversal.
	\end{definition}
	
	Let $X$ be a smooth variety and $C\subseteq T^*X$ a closed conical subset of the cotangent bundle.
	
	\begin{definition}
		A test pair $(h,f)$ is called $C$-transversal if $Y,U$ are also smooth, and for every geometric point $u\rightarrow U$, $h$ is $C$-transversal at $u$ and $f$ is $h^{\circ}C_u$-transversal at $u$.
	\end{definition}
	
	For the definition of universal local acyclicity, see \cite[\S 1.2]{Barrett2023}.
	
	\begin{definition}
		Let $X$ is a smooth variety and $\f$ a constructible sheaf. A $C$-transversal pair is called $\f$-acyclic if $f$ is universally locally acyclic with respect to $h^*\f$.
		
		We say that $\f$ is \emph{microsupported} on $C$ if every $C$-transversal test pair is $\f$-acyclic.
	\end{definition}
	
	The existence of a minimal conical subvariety as in the next definition is \cite[Theorem 1.5]{Beilinson2016} for sheaves with torsion coefficients and \cite[Theorem 1.5]{Barrett2023} in general.
	
	\begin{definition}\label{ess}
		For any smooth variety $X$ and \'etale constructible sheaf $\f\in D(X)$, the \emph{singular support} $SS(\f)$ of $\f$ is defined to be the minimal conical closed subvariety $SS(\f)\subseteq T^*X$ of the cotangent bundle such that $\f$ is microsupported on $SS(\f).$
	\end{definition}	
	
	We recall certain functorial properties of the singular support that we will need. They are essentially the same as in the characteristic $0$ case.
	
	If $r: X
	\rightarrow Z$ is a map of smooth varieties and $C$ is a closed conical subset of $T^*X$
	whose base is proper over $Z$, $r_{\circ}C$ is defined as the image of $dr^{-1}(C)\subseteq T^*Z \times_Z X$ by the projection $T^*Z \times_Z X\maps T^*Z.$ It is a closed conical subset of $T^*Z$.
	
	\begin{proposition}\label{ssprop}
		Let $\f\in D(X)$ be a constructible sheaf on a smooth variety $X$. Then, the following are true.
		\begin{enumerate}[label=(\roman*)]
			\item The singular support of $\f$ is the union of the singular supports of its irreducible constituents.
			\item Let $h:X\rightarrow Y$ be a smooth morphism. Then $SS(h^*\f)=h^{\circ}SS(\f).$
			\item Let $r:X\rightarrow Z$ is a proper morphism. Then
			$SS(r_*\f)\subseteq r_{\circ}SS(\f).$
		\end{enumerate}	  
	\end{proposition}
	\begin{proof}
		The first property is \cite[Theorem 1.5.(viii)]{Barrett2023}.
		The second property is \cite[Theorem 1.5.(x)]{Barrett2023}, see also \cite[Lemma 2.2]{Beilinson2016}.
		
		The third property follows from \cite[Lemma 4.4.(ii)]{Barrett2023} since $\f$ is microsupported on $SS(\f)$ by definition.
	\end{proof}
	
	\begin{remark}
		The first property is more generally true for $SS(\f)$-transversal morphisms, but a smooth morphism is transversal to any closed conical subset. The second property is true when the base of $SS(\f)$ is proper over $Z$. 
		
		The inequality in the second property is strict: For an example, consider the pushforward of the constant sheaf on $\mathbb{A}^1$ by the Frobenius.
	\end{remark}
	
	\subsection{Singular support of induction}
	
	To prove Theorem \ref{main}, we also need the equivalent of \cite[Lemma 1.2]{Mirkovic1988} for characteristic $p$. 
	
	Let $A$ be a connected algebraic group acting on a variety $X$. For any connected subgroup $B$, consider the following commutative diagram.
	\[
	\begin{tikzcd}
		A\times X \arrow[r,"\nu"]\arrow[d,"a"] & A/B\times X\arrow[d,"p"]\\
		X & X
	\end{tikzcd}
	\]
	given by
	\[
	\begin{tikzcd}
		(a,x) \arrow[r,"\nu"]\arrow[d,"a"] & (aB,x)\arrow[d,"p"]\\
		a^{-1}\cdot x & x.
	\end{tikzcd}
	\]
	For a $B$-equivariant constructible sheaf $\f\in D_B(X)$, there exists a unique $\tilde{\f}\in D_A(A/B\times X)$ such that $a^*\f=\nu^*\tilde{\f}$. 
	
	\begin{definition}\label{ind2}
		The functor $\Gamma_B^A : D_B(X)\maps D_A(X)$ defined by $\f=p_*\tilde{\f}\in D_A(X)$ is called the {\em induction functor}.
	\end{definition}
	
	\begin{lemma}\label{indbound}
		For $\f\in D_B(X)$, we have 
		$$SS(\Gamma_B^G \f)\subseteq G \cdot SS(\f).$$
	\end{lemma}
	
	\begin{proof}
		Let $pr_2$ be the second projection from either $T^*(G\times X), T^*(G/B\times X)$ to $T^*X$ as in \cite[Proof of Lemma 1.2]{Mirkovic1988}. Observe that by definition $p_{\circ}=pr_2$. By Proposition \ref{ssprop}.(ii), $pr_2(SS(a^*\f))=G\cdot SS(\f)$, as in \emph{loc. cit.}
		
		By Proposition \ref{ssprop}.(ii) and the above observation, $p_{\circ}(SS(\tilde{\f}))=pr_2(SS(\nu ^*\tilde{\f}))=pr_2(SS(a^*\f))=G\cdot SS(\f)$
		
		We have that $G/B$ is a flag variety and therefore proper, so $p$ is proper being the base change of a proper morphism. By Proposition \ref{ssprop}.(iii), $$SS(\Gamma_B^G \f)=SS(p_* \tilde{\f})\subseteq p_{\circ}(SS(\tilde{\f}))=G\cdot SS(\f).$$ 		
	\end{proof}
	
	\begin{remark}
		There is no need to consider the closure as in \cite[Lemma 1.2]{Mirkovic1988}, because we have restricted to the case $G/B$ is already proper.
	\end{remark}
	
	\section{Tame perverse sheaves}
	
	\subsection{Definition of a tame perverse sheaf}
	
	Let $X$ be a smooth variety over an algebraically closed field $k$. All sheaves we consider are assumed to be $l$-adic for $l\neq \cha(k)$.  
	
	We denote by the letter $E$ any constant sheaf. Let $U\subseteq X$ be a locally closed subset. We recall a well-known definition.
	
	\begin{definition}\label{fm}
		A local system $\ls$ on $U$ has finite monodromy if there exists a finite \'etale covering map $f:\tilde{U}\rightarrow U$ such that $f^* \ls=E$.
	\end{definition}
	
	We also recall the following \cite[\S 1]{Kerz2010}.
	
	\begin{definition}
		Let $\bar{C}$ be a proper, connected and regular curve of finite type over $\Spec(k)$ and $C\subseteq \bar{C}$ an open subscheme. Every point $x\in \bar{C}\setminus C$ defines a valuation $v_x$ on $k(C)$. An \'etale covering of curves $C'\rightarrow C$ is called {\em tame} if for every $x\in \bar{C}\setminus C$ the valuation $v_x$ is tamely ramified in $k(C')\mid k(C).$
	\end{definition}
	
	Finally, we recall one of the several equivalent definitions for tameness of an \'etale covering in higher dimensions in \cite[\S 4]{Kerz2010}, see also \cite[Definition 5.1]{Huebner2021}. 
	
	\begin{definition}\label{tc}
		An \'etale covering $Y\rightarrow X$ is called {\em tame} if for every morphism $C\rightarrow X$ with $C$ regular, the base change $Y\times_X C\rightarrow C$ is tame.
	\end{definition}
	
	\begin{definition}\label{stcd}
		Let $X=\Spec(A)$ be an smooth variety of dimension $n$ and $\pi_1,\ldots,\pi_n$ a local system of parameters around a point $x\in X$ such that $D=V(\pi_1)\cup \ldots \cup V(\pi_n)$ is a simple normal crossings divisor. Let $X'=X\setminus D\cong \Spec A'$. 
		
		We call a covering of the form 
		$$f: \Spec A'[T_1,\ldots, T_n]/(T_1^{k_1}-\pi_1,\ldots, T_n^{k_n}-\pi_n)\rightarrow \Spec A'$$
		such that $p\nmid k_i$ for any $i=1,\ldots n$ a {\em standard tame covering.}
	\end{definition}
	
	\begin{example}\label{stc}
		Standard tame coverings are tame \'etale coverings. For the case of the complement of a simple normal crossings divisor, these are essentially all of them, see \cite[Theorem 4.4]{Kerz2010} and \cite[Expose XIII, Proposition 5.1]{Grothendieck1971}.
	\end{example}
	
	The following definitions are motivated by Definition \ref{tc}.
	
	\begin{definition}\label{tls}
		A local system $\ls$ with finite monodromy will be called {\em tame} if the \'etale covering $f: \tilde{U}\rightarrow U$ in Definition \ref{fm} can be taken to be tame. 
	\end{definition}
	
	\begin{definition}\label{tps}
		An irreducible perverse sheaf $IC(U,\ls)$ on a smooth variety $X$ is called {\em tame} if $\ls$ is a tame local system on $U$. A perverse sheaf on $X$ is called {\em tame} if all its irreducible constituents are tame.
	\end{definition}

	\subsection{Properties}
	
	In this subsection we show properties of tameness that will be used later.
	
	\begin{lemma}\label{abel}
		The category $\Perv_{\tame}(X)\subseteq \Perv(X)$ consisting of the tame perverse sheaves and the morphisms between them is an abelian subcategory of the category of perverse sheaves and is stable under extensions.
	\end{lemma}
	
	\begin{proof}
		Trivial by checking on the level of irreducible constituents.
	\end{proof}
	
	The following is in \cite{Kerz2010}, but we provide a proof for completeness. 
	
	\begin{lemma}\label{tambc}
		Tameness of an \'etale covering is stable under arbitrary base change.
	\end{lemma}
	
	\begin{proof}
		Let $f:Y\rightarrow X$ be a tame \'etale covering and $g:Z\rightarrow X$ a morphism. Then for an arbitrary morphism $C\rightarrow Z$ where $C$ is a regular curve, we get the following diagram
		\[
		\begin{tikzcd}
			Y\times_X Z\times_Z C \arrow[r]\arrow[d]& C\arrow[d]\\
			Y\times_X Z\arrow[r]\arrow[d] & Z\arrow[d]\\
			Y\arrow[r] & X
		\end{tikzcd}	
		\]
		It is enough to prove that the upper morphism of curves is tame. This follows by $Y\times_X Z\times_Z C\cong Y\times_X C$ and the tameness of $Y\rightarrow X$.
	\end{proof}
	
	For a smooth morphism $f:X\rightarrow Y$, we denote by $f^{\dagger}$ the functor $f^*[d]$ as in \cite{Achar2021}. We remark that Mirkovic and Vilonen use $f^{\circ}$ for the same functor. 
	
	\begin{lemma}\label{tampb}
		Let $f:X\rightarrow Y$ be a smooth morphism with connected fibers, and $\f$ a tame perverse sheaf on $Y$. Then $f^{\dagger} \f$ is tame.
	\end{lemma}
	
	\begin{proof}		
		Since $f^{\dagger}$ is t-exact by \cite[\S 4.2.4]{Beilinson1982}, we can assume $\f$ is irreducible. We write it as $\f=IC(Y_0,\ls)$. Since smoothness is preserved by arbitrary base change, we can also restrict $f$ to the preimage $f^{-1}(\overline{Y_0})$. 
		
		Then, $f^{\dagger}\f\cong IC(f^{-1}(Y_0),f^*\ls)$, and irreducible by \cite[Theorem 3.6.6]{Achar2021}. For $\ls$ by Definition \ref{tls} there must be a tame \'etale covering $g:U\rightarrow Y_0$ such that $g^* \f=E.$ Therefore, by the diagram
		\[
		\begin{tikzcd}
			U\times_{Y_0} X \arrow[r,"g'"]\arrow[d,"f'"]& X\arrow[d,"f"]\\
			U\arrow[r,"g"] & Y_0
		\end{tikzcd}	
		\]
		since by compatibility of pullbacks with composition, $g'^* f^*\ls=f'^*g^*\ls=f'^* E=E$, $g'$ is an \'etale covering trivializing $f^*\ls$. By Lemma \ref{tambc}, $g'$ is also tame, therefore $f^{\dagger}\f$ is tame.
	\end{proof}
	
	Notice that, in general, pushforwards do not preserve tameness even under very strong assumptions.
	
	\begin{example}\label{pushex}
		Consider a projective wildly ramified covering of curves $f:X\rightarrow Y$. This decomposes as $X\xrightarrow[]{i} Y \times \pp^1 \xrightarrow[]{p} Y$. $\f=i_*E$ is a tame sheaf, while $p_* \f=f_*E$ is not. We remark that $p$ is a smooth proper map.
	\end{example}
	
	On a torus, tame local systems coincide with Kummer local systems. Indeed, for $n\in \mathbb{N},$ let $n:T\rightarrow T$ be the $n$-th power isogeny $n(t)=t^n$.
	
	\begin{lemma}\label{tameq}
		For a torus $T$, a local system $\ls$ is tame if and only if $n^*\ls\cong E$ for some $n$ coprime to $p$.
	\end{lemma}
	
	\begin{proof}
		Let $R=k[x_1^{\pm},\ldots,x_k^{\pm}]$ and $T=\Spec R$.
		Since $n^*\ls=E$, $\ls$ trivializes under the covering $n:\Spec R[t_1,\ldots, t_k]/(t_1^n-x_1,\ldots, t_k^n-x_k)\rightarrow \Spec R,$ which is a standard tame covering since $(n,p)=1$, see Definition \ref{stcd}. 
		
		If $\ls$ is tame, then it must trivialize under some tame covering $f$. Since $T=\mathbb{A}^n\setminus D$ where $D=\{x_1=0,\ldots, x_n=0\}$, $f$ can be taken to be a standard tame covering by Example \ref{stc}, therefore of the form $f:\Spec R[t_1,\ldots, t_k]/(t_1^{n_1}-x_1,\ldots, t_k^{n_k}-x_k),$ where $(n_i,p)=1$ for all $i=1,\ldots, k.$ Defining $n=\prod n_i$, we see that we can write $n$ as a composition $g\circ f$, and therefore $n$ also trivializes $\ls$, since $n^*\ls=(g\circ f)^* \ls=g^*E=E$. 
	\end{proof}
	
	We show that conormality of the singular support still holds in characteristic $p$ for tame perverse sheaves upon conditions on the geometry of the ramification divisor.
	
	\begin{definition}
		Let $X$ be a variety and $D\subset X$ a closed subvariety. A map $f:\tilde{X}\rightarrow X$ will be a called a uniform resolution of $D$ if
		\begin{enumerate}
			\item $\tilde{X}$ is smooth and $f$ is proper and birational.
			\item $f^{-1}(D)$ is a simple normal crossings divisor in $\tilde{X}$.
			\item $f$ is a stratified submersion on tangent spaces for the stratification $D=\bigsqcup_{i\in I} D_i$ by smooth strata of codimension $i$. 
		\end{enumerate}
	\end{definition}

	\begin{lemma}\label{conormality}
		Let $X$ be a smooth variety over an algebraically closed field $k$. If $\f$ is a tame perverse sheaf on $X$ such that for every irreducible constituent $IC(U,\ls)$ we have that $D:=\bar{U}\setminus U$ has a uniform resolution, then $SS(\f)$ is conormal to its base.
	\end{lemma}
	
	\begin{proof}
		Assume without loss of generality that $\f:=IC(U,\ls)$ is irreducible. If $D$ is already a simple normal crossings divisor, the assertion follows from \cite[Lemma 3.3]{Saito2016}.
		
		For the general case, shrink $U$ by intersecting with the open locus where $f$ is an isomorphism. Consider the diagram
		
		\[
		\begin{tikzcd}
			f^{-1}(U) \arrow[r,"\tilde{j}"]\arrow[d,"g"]& \tilde{X}\arrow[d,"f"]\\
			U\arrow[r,"j"] & X
		\end{tikzcd}	
		\]
		where $j,\tilde{j}$ are inclusions and $g$ is the restriction of $f$ to $f^{-1}(U)$.
		
		We define $\cG:=\tilde{j}_!g^*\ls,$ and notice that $$j_!\ls\cong f_*f^*j_!\ls\cong f_*\cG$$ 
		where the first isomorphism is true because $f$ is an isomorphism over $U$, and the second by base change. Therefore, by \cite[Lemma 2.2]{Beilinson2016}
		
		$$SS(j_!\ls)= SS(f_*\cG)\subseteq f_{\circ}SS(\cG),$$
		and $SS(\cG)$ is conormal since $f^{-1}(D)$ is a simple normal crossings divisor and $g^*\ls$ is tame.
		
		Assume $SS(\f)$ was not conormal. Then there exists $z\in D_i$ and $w$ not conormal to $D_i$, such that $w\in SS(\f).$ Therefore, there exists $v$ tangent to $D_i$ such that $\langle w,v\rangle\neq 0.$ By property (3), there exists $v'$ tangent to $f^{-1}(D_i)$ such that $\langle df(w),v'\rangle\neq 0,$ and by conormality of $SS(\cG)$ we have a contradiction.
	\end{proof}
	
	\begin{remark}
		For the case of character sheaves treated in this paper we need conormality for the Bruhat stratification, which admits the Bott-Samelson resolution. In particular, let $w\in W$ an element of the Weyl group and $Y_w=\overline{BwB}$. Choose a reduced expression $w=s_1\ldots s_n.$ Then the Bott-Samelson resolution is
		$$Y_{s_1}\times_B \cdots \times_B Y_{s_n}\rightarrow Y_w$$ 
		The Bott-Samelson resolution is uniform by $B$-equivariance.
	\end{remark}
	
	\begin{remark}
		For a counterexample to conormality when $\f$ is not tame, see \cite[Example 1.6]{Beilinson2016}.
	\end{remark}
	
	\section{Proof of the main theorem}
	
	We adapt Mirkovic and Vilonen's proof to show Theorem \ref{main}. 
	
	First, we recall the definition of character sheaves following \cite[\S 4, \S 5]{Mars1989}. Let $G$ be a connected reductive group over a field $k$ of characteristic $0$, and $T$ a maximal torus of $G$. Fix a Borel $B=TU$ containing $T$ with unipotent radical $U$. Let $\Phi$ be the set of roots of $G$ and $\Delta\subseteq \Phi$ the choice of positive roots defined by $B$. 
	
	Every root $a\in \Phi$ provides us with an one-parameter subgroup $x_a:k\maps G$. Let $X_a:=im(x_a)$ denote the image.
	
	For an element $w\in W$, we define $R_w:=\{a\in \Phi \mid a\in \Delta, -w^{-1}a\in \Delta \}.$ We denote by $U_w\subseteq U$ the subgroup of $U$ generated by $T$ and the $X_a$ for $a\in R_w.$
	
	We fix representatives $\dot{w}\in N_G(T)/T$ for the elements $w\in W$. By Bruhat decomposition, $G=\bigsqcup_{w\in W} B\dot{w}B$. We write $G_w:= B\dot{w}B$ and we remark as in \cite[\S 4.1.1]{Mars1989} that the map $U_w\times T \times U \maps G_w$ defined by $(u,t,u')\mapsto u\dot{w}tu'$ is an isomorphism of varieties, which provides us with a projection $pr: G_w\rightarrow T$ by setting $pr(u\dot{w}tu')=t$.
	
	For a Kummer local system $\ls$ on $T$, we define $\ls_{ w}:=pr^*\ls$, and $A_{w}^{\ls}:=IC(G_w,\ls_{w}).$ Then, as in \cite[\S 2.1]{Mirkovic1988}, we can define character sheaves in terms of the induction functor $\Gamma_B^G: D_B(G)\maps D_G(G)$. We say that a Kummer local system $\ls$ on $T$ is fixed by $w$ if $w^*\ls=\ls.$ Recall that for a Kummer local system $\ls$ fixed by $w\in W$, $A_{w}^{\ls}$ is a $B$-equivariant perverse sheaf on $G$ by \cite[Lemma 4.1.2]{Mars1989}.
	
	For a semisimple complex, we call irreducible constituents the simple perverse sheaves appearing as shifts of direct summands.
	
	\begin{definition}\label{cs}
		Character sheaves are the irreducible constituents of the complexes $K_w^{\ls}:=\ind(A_w^{\ls})$ for all choices of $w\in W$ and a Kummer local system $\ls$ on $T$ fixed by $w$.
	\end{definition}
	
	Finally, we prove Theorem \ref{main}.
	
	\begin{proof}
		Let $\f$ be a character sheaf on $G$. By Definition \ref{cs}, there exists a Kummer local system $\ls$ on $T$ and an element $w\in W$ such that $\ls$ is fixed by $w$ and $\f$ is an irreducible constituent of $K_w^{\ls}$. By Proposition \ref{ssprop}.(i), $SS(\f)\subseteq SS(K_w^{\ls})$. By Lemma \ref{indbound}, $SS(K_w^{\ls})=SS(\ind A_w^{\ls})\subseteq G\cdot SS(A_w^{\ls})$. Since $\nn$ is $G$-equivariant, it is enough to show $SS(A_{w}^{\ls})\subseteq G\times \nn$.
		
		Since $\ls$ is Kummer, $\ls$ is tame by Proposition \ref{tameq}, and therefore $\ls_{w}$ is tame by Proposition \ref{tampb}. $\ls_w$ is an irreducible local system since $\ls$ is, and therefore by Definition \ref{tps} $A_{w}^{\ls}$ is a tame perverse sheaf. 	
		
		As in \cite[Proof of Theorem 2.7]{Mirkovic1988}, the fiber of the conormal bundle at $g$ of a Bruhat cell $BgB$ is $n \cap n^g$ where $n$ is the Lie algebra of $U$, and $n^g:=gng^{-1}$.
		Since the conormal bundles to Bruhat cells are nilpotent and $A_{w}^{\ls}$ is tame, we have $SS(A_{w}^{\ls})\subseteq G\times \nn$ by Lemma \ref{conormality} applied to the Bott-Samelson resolution.
	\end{proof}

	\printbibliography
	
\end{document}